\documentclass[11pt,a4paper]{amsart}

\usepackage{geometry}
\usepackage[T1]{fontenc}
\usepackage{lmodern}
\usepackage[utf8]{inputenc}

\usepackage{mathtools,amssymb,bm,array,booktabs,xcolor}

\newcommand*\bQ{\mathbb{Q}}
\newcommand*\bZ{\mathbb{Z}}
\newcommand*\bN{\mathbb{N}}
\newcommand*\bR{\mathbb{R}}
\newcommand*\la{\langle}
\newcommand*\ra{\rangle}
\newcommand*\eqdef{\coloneqq}

\newcommand*\divides{\mid}
\newcommand*\ndivides{\nmid}
\newcommand*\Rplus{\protect\hspace{-.1em}\protect\raisebox{.35ex}{\smaller{\smaller\textbf{+}}}}

\newcommand*\Pmod[1]{\mathclose{}\ ({\operatorname{mod}}\ #1)}
\newcommand*\defined[1]{\textsl{#1}}

\newcommand*\Legfrac[2]{\bigl(\!\frac{\,#1\,}{\,#2\,}\!\bigr)}
\newcommand*\nofrac[2]{#1/#2}

\DeclarePairedDelimiterX\set[2]\lbrace\rbrace{\,#1\mathclose{}:\mathopen{}#2\,}

\usepackage{amsthm}

\newtheorem{lemma}{Lemma}
\newtheorem{proposition}[lemma]{Proposition}

\newtheorem{theorem}[lemma]{Theorem}
\newtheorem{conjecture}[lemma]{Conjecture}

\theoremstyle{definition}

\edef\TODAY{{\noexpand\color{blue}\uppercase{\today}}}

\title{Ternary quadratic forms representing a~given~arithmetic~progression}

\usepackage{hyperref}

\def\SA#1#2{Sage}

\def\MA#1#2{Magma}

\begin{document}

\author{Tom\' a\v s Hejda}
\address{Charles University, Faculty of Mathematics and Physics, Department of Algebra, Sokolov\-sk\' a 49/83, 18675 Praha~8, Czech Republic
\newline\indent
University of Chemistry and Technology, Prague, Department of Mathematics, Studentská~6, 16000 Praha~6, Czech Republic}
\email{tohecz@gmail.com}

\author{V\' \i t\v ezslav Kala}
\address{Charles University, Faculty of Mathematics and Physics, Department of Algebra, Sokolov\-sk\' a 49/83, 18675 Praha~8, Czech Republic}
\email{kala@karlin.mff.cuni.cz}

\thanks{\textit{Funding.} This work was supported by the Czech Science Foundation (GA\v CR) [grant numbers 21-00420M, 17-04703Y];
	and Charles University [grant numbers PRIMUS/20/SCI/002, UNCE/SCI/022].}



\begin{abstract}
A positive quadratic form is $(k,\ell)$-universal if it represents all the numbers $kx+\ell$ where $x$ is a non-negative integer,
 and almost $(k,\ell)$-universal if it represents all but finitely many of them.
We prove that for any $k,\ell$ such that $k\nmid\ell$ there exists an almost $(k,\ell)$-universal diagonal ternary form.
We also conjecture that there are only finitely many primes $p$ for which a $(p,\ell)$-universal diagonal ternary form exists (for any $\ell<p$)
 and we show the results of computer experiments that speak in favor of the conjecture.
\end{abstract}

\maketitle

\section{Introduction}

The sum of three squares $x^2+y^2+z^2$ does not represent any integer of the form $8n+7$ and similarly every other positive ternary quadratic form fails to represent some arithmetic sequence.
Conversely, around 1797 Legendre showed that $x^2+y^2+z^2$ represents all positive integers that are not of the form $4^k(8n+7)$ (with $k,n\geq 0$) and there have been numerous results concerning the integers represented by ternary quadratic forms. Before discussing some of them, let us introduce some basic notions.

A \defined{positive ternary quadratic form} is a form $Q(x,y,z)=ax^2+by^2+cz^2+dyz+exz+fxy$, where $a,b,c,d,e,f$ are integers and $Q(x,y,z)>0$ for all real numbers $x,y,z$, not all of them $0$.

For positive integers $k, \ell$ consider the arithmetic sequence
\[
	S_{k, \ell}\eqdef\set{kx+\ell}{x\in\bZ,\, x\geq 0}
.\]
We say that a positive quadratic form with $\bZ$-coefficients is \defined{$(k, \ell)$-universal} if it represents all elements of 
$S_{k, \ell}$ over the ring of integers $\bZ$.
A quadratic form is  \defined{almost $(k, \ell)$-universal} if it represents almost all elements of 
$S_{k, \ell}$, i.e., if there are at most finitely many elements of $S_{k, \ell}$ that are not represented.

Kaplansky~\cite{kaplansky_1995} showed that there are at most 23 ternary forms that represent all odd positive integers (i.e., that are $(2,1)$-universal) and proved the $(2,1)$-universality of 19 of them. Jagy~\cite{jagy_1996} dealt with one of the remaining candidates and, assuming the $(2,1)$-universality of the 3 other forms, 
Rouse~\cite{rouse_2014} proved the $451$-theorem:
A positive quadratic form (of any rank) is $(2,1)$-universal if and only if it represents all the integers $1,3,5,\dots,451$.

Oh~\cite{oh_2011_ijnt} then proved that for any $(k,\ell)$, there are only finitely many equivalence classes of $(k,\ell)$-universal ternaries (that are moreover classical, i.e., $d,e,f$ are all even).

This was followed by investigations of all the $(k,\ell)$-universal diagonal ternary forms for small values of $k$ and $1\leq\ell<k$. Independently, Pehlivan and Wil\-liams~\cite{pehlivan_williams_2018} computed all such possible candidates for $k\leq 11$, and Sun~\cite{sun_2017} found all such candidates with $k\leq 30$.
Pehlivan and Williams established the $(k,\ell)$-universality of a number of their candidates and then Wu and Sun~\cite{wu_sun_2018} proved this for more of these forms.

Notably, when $k=p\geq 11$ is a prime, then no $(k,\ell)$-universal diagonal ternary forms appears in these lists!

$(k,\ell)$-universal forms are intimately connected to regular forms, i.e., quadratic forms that represent over $\mathbb Z$ all the integers that they represent over $\mathbb R$ and the ring $\mathbb Z_p$ of $p$-adic integers for all primes $p$.
Jagy, Kaplansky, and Schiemann~\cite{jagy_kaplansky_shiemann_1997}
proved that there are at most 913~regular ternary forms and established the regularity of all but 22~of them. Oh~\cite{oh_2011_aa}
then proved the regularity of 8 of these, and then Lemke Oliver~\cite{lemke_2014} dealt with the 14~remaining cases under the assumption of Generalized Riemann Hypothesis,
using the method of Ono and Soundrajaran~\cite{ono_soundararajan_1997}.

The problem of precisely determining the set of integers represented by a given ternary quadratic form is still open, although it has been thoroughly studied. Let us mention only the results by 
Kneser~\cite{kneser_1961},
Duke and Schulze-Pillot~\cite{duke_schulze-pillot_1990},
Earnest, Hsia, and Hung~\cite{earnest_hsia_hung_1994}, and point the interested reader to 
the very nice survey by Hanke~\cite{hanke_2004}.

\bigskip

In this short paper, we study $(k,\ell)$-universality of diagonal ternary quadratic forms.
Considering \emph{almost} $(k,\ell)$-universal diagonal ternaries, we show that they always exist.

\begin{theorem}\label{thm:main-almost}
	Let $k, \ell$ be positive integers such that $k\nmid\ell$. Then there is a diagonal ternary positive quadratic form that is almost $(k, \ell)$-universal.	
\end{theorem}

We prove the theorem in \S\,\ref{sec:almost} by first dealing with almost $(p, \ell)$-universal forms; in fact, we show that for each prime $p$, there is a prime $q$ (or $q=1$) such that the form $x^2+qy^2+pz^2$ is anisotropic precisely at $p$ (and $\infty$), and that this form is then almost $(p, \ell)$-universal. This then quickly implies the theorem for general $k$.

\bigskip

In \S\,\ref{sec:univ} we expand on the observation (based on the results of Pehlivan and Wil\-liams~\cite{pehlivan_williams_2018} and Sun~\cite{sun_2017}) that when $p$ is a prime satisfying $11\leq p\leq 29$, then there is no $(p,\ell)$-universal diagonal ternary form for $1\leq \ell<p$.

We first search for $(p,\ell)$-universal diagonal ternaries and obtain that, for $11\leq p\leq 1237$, the only case when they can exist is $(101, 98)$, when $x^2+2y^2+101z^2$ appears to be $(101, 98)$-universal.

To obtain more refined understanding of the situation, we then consider the number of ``gaps'' of a given form (that satisfies the necessary anisotropy conditions), i.e., of (small) integers that are not represented, see \S\,\ref{sec:exp-ne}. 
Our computations suggest that the number of gaps is always larger than $p\log p$, which provides heuristic argument in favor of the following conjecture (details are discussed in~\S\,\ref{sec:univ}).

\begin{conjecture}\label{conj:fin}
There are only finitely many primes $p$ and $1\leq\ell<p$ possessing a diagonal ternary positive $(p,\ell)$-universal quadratic form.
\end{conjecture}

In fact, the data suggest even the stronger conjecture that
the form $x^2+2y^2+101z^2$ was the last missing one and that now the knowledge is exhaustive:

\begin{conjecture}\label{conj:2-101}
The following table gives the complete list of diagonal ternary positive $(p,\ell)$-universal quadratic forms for a prime $p$ and  $1\leq\ell<p$
 (here $\la a,b,c\ra$ stands for the form $ax^2+by^2+cz^{2\!}$):
\[
	\begin{tabular}{*3{>$l<$}}\toprule
	p & \ell & \text{$(p,\ell)$-universal forms} \\
	\midrule[\heavyrulewidth]
	2
	& 1 & \la1,1,2\ra,\, \la1,2,3\ra,\, \la1,2,4\ra \\
	\midrule
	3
	& 1 & \la1,1,3\ra,\, \la1,1,6\ra,\, \la1,3,3\ra,\, \la1,3,9\ra,\, \la1,6,9\ra \\
	& 2 & \la1,1,3\ra,\, \la1,1,6\ra,\, \la2,3,3\ra \\
	\midrule
	5
	& 1 & \la1,2,5\ra,\, \la1,5,10\ra \\
	& 2 & \la1,2,5\ra \\
	& 3 & \la1,2,5\ra \\
	& 4 & \la1,2,5\ra,\, \la1,5,10\ra \\
	\midrule
	7
	& 1 & \la1,2,7\ra,\, \la1,7,14\ra \\
	& 2 & \la1,2,7\ra \\
	& 3 & \la1,2,7\ra \\
	\midrule
	101
	& 98 & \la1,2,101\ra \\
	\bottomrule
	\end{tabular}
\]
\end{conjecture}

The universality of the forms for $p=2,3,5$ has been established (see, e.g., \cite{pehlivan_williams_2018}),
 whereas for $p=7,101$ it is only conjectural.

Note that the forms $\la1,2,p\ra$ appear frequently in the preceding table.
This is not an accident, as it seems that these forms are the most likely candidates for $(p,\ell)$-universality 
(see \S\,\ref{sec:exp-ne}, \ref{sec:heur}). In fact, it turns out that the heuristic argument that we use for dealing with the other forms fails in this case! Hence we have to consider  the forms $\la1,2,p\ra$ (together with $\la1,1,p\ra$ and $\la1,3,p\ra$) separately in detail in \S\,\ref{sec:12p}.

The source codes are available at \url{https://github.com/tohecz/ternary}.

It is of course very interesting to consider the existence of $(p,\ell)$-universal ternaries without the restriction that $\ell<p$ and without the diagonality assumption. 
As we have shown that almost $(p,\ell_0)$-universal ternaries always exists, it trivially follows that also $(p,\ell)$-universal ternaries exist once $\ell\equiv\ell_0\Pmod p$ is sufficiently large. Nevertheless, if we set a bound $\ell<Cp$ for a fixed positive integer $C$, our heuristics suggest that there again should be only finitely many $(p,\ell)$-universal diagonal ternaries with $\ell<Cp$. 
We have not done almost any computations with non-diagonal forms. Nevertheless, our (mostly unfounded) guess might be that there are only finitely many $(p,\ell)$-universal \emph{non-diagonal} ternaries  as well.

\section{Existence of almost \texorpdfstring{$(k,\ell)$}{(k,l)}-universal forms}\label{sec:almost}

In the rest of the article we will consider only diagonal ternary positive forms, i.e., quadratic forms $Q(x,y,z)=ax^2+by^2+cz^2=:\langle a,b,c\rangle$, where $a,b,c$ are positive integers.
We denote $d_Q=abc$ the determinant of $Q$ (note that the determinant of a ternary form is sometimes defined as $8abc$; in our definition we are following \cite{cassels_1978}).

Let $\nu$ be a place of $\bQ$. 
A quadratic form $Q(x,y,z)$ is \defined{isotropic at $\nu$} if it non-trivially represents $0$ over the completion $\bQ_\nu$, i.e., if $Q(x,y,z)=0$ for some $x,y,z\in\bQ_\nu$, not all of them $0$. Otherwise $Q$ is \defined{anisotropic at $\nu$}.
Note that if $\nu=p$ is a finite place corresponding to a prime $p$, then $Q=\la a,b,c\ra$ can be anisotropic at $p$ only if $p\mid 2abc$.

A positive ternary form is always anisotropic at $\infty$ and, by Hilbert reciprocity law, it is anisotropic at an odd number of finite places $\nu$.

\begin{proposition}\label{prop:anisotropic}
	Let $p$ be a prime. Then there is a positive diagonal ternary form $Q=\la 1,q,p\ra$ that is anisotropic precisely at $p$ and $\infty$, where $q=1$ or a prime different from $p$.	
\end{proposition}

\begin{proof}
	For $p=2$ we can take $Q=\la 1,1,2\ra$, so let us assume that $p$ is an odd prime. We will distinguish several cases according to the value $p\Pmod 8$.
	
	Case $p\equiv 3\Pmod 4$. Then we claim that $Q=\la 1,1,p\ra$ works. The only candidates for primes at which $Q$ can be anisotropic are $2$ and $p$. By Hilbert reciprocity, it suffices to show that $Q$ is anisotropic at $p$. Assume that $x^2+y^2+pz^2=0$ with $x,y,z\in\bZ_p$. Then $x^2+y^2\equiv 0\Pmod p$, and so $x\equiv y\equiv 0\Pmod p$, because $-1$ is a quadratic non-residue modulo $p$ in this case. But then also $z\equiv 0\Pmod p$ and we have 
	$(\nofrac xp)^2+(\nofrac yp)^2+p(\nofrac zp)^2=0$. Continuing in this way, we get that $x=y=z=0$, i.e., that $Q$ is anisotropic at $p$.
	
	Case $p\equiv 5,7\Pmod 8$. In this case the form $Q=\la 1,2,p\ra$ works, which can be proved by the same argument as in the previous paragraph.
	
	Case $p\equiv 1\Pmod 8$. We will show that the form  $Q=\la 1,q,p \ra$ works if $q$ is a prime such that $q\equiv 3\Pmod 4$ and the Legendre symbol $\Legfrac qp =-1$ (such primes $q$ clearly exist).	
	For this form $Q$, the anisotropic candidates are $2$, $p$, and~$q$. Distinguishing the two possibilities for $q\Pmod 8$, it is easy to verify that the form $Q$ is always isotropic at 2. Hence it suffices to show that $Q$ is anisotropic at $p$.	
	As $p\equiv 1\Pmod 8$, $-1$ is a quadratic residue modulo $p$, and so $\Legfrac {-q}p =-1$. Hence $x^2+qy^2\equiv 0\Pmod p$ implies that $x\equiv y\equiv 0\Pmod p$ and we see as before that $Q$ is indeed anisotropic at $p$.	
\end{proof}

\begin{proposition}\label{prop:prime-almost-univ}
	Let $a,b,v$ be positive integers and $p$ an odd prime such that $abv$ is squarefree and $p\nmid abv$. Let $Q=\la a,b,vp\ra$ be a positive diagonal ternary form that is anisotropic precisely at $p$ and $\infty$. Then $Q$ is almost $(p,\ell)$-universal for every positive integer $\ell$ such that $p\nmid\ell$.
\end{proposition}

\begin{proof}
	We will use a theorem of Duke and Schulze-Pillot~\cite{duke_schulze-pillot_1990}, cf.~\cite[Theorem on p.~11]{hanke_2004}. For any undefined notions in the proof, see, e.g.,~\cite{hanke_2004}.
	
	We are interested in almost $(p,\ell)$-universality and we have $p\nmid\ell$, so we need to show that $Q$ locally represents all elements of the corresponding arithmetic progression and that there are no spinor exceptions.
	
	The local representation is no problem at $\bR$ and at the isotropic places, so we need to check it only at the anisotropic place $p$. It suffices to show that the binary form $ax^2+by^2$ with $p\nmid q$ represents all non-zero classes modulo $p$. As $a$ is invertible modulo $p$, let $e\equiv ba^{-1}\Pmod p$; it then suffices to show that $x^2+ey^2$ represents all non-zero classes modulo $p$.
	
	If $e$ is a quadratic non-residue modulo $p$, then $ey^2$ represents all the non-residues and $x^2$ represents all the residues. If $e$ is a quadratic residue, then $x^2+ey^2$ represents the same elements modulo $p$ as $x^2+y^2$. But the latter form represents (over $\bZ$) all the primes $\equiv 1\Pmod 4$, which cover all the non-zero classes modulo $p$.
	
	Let us now consider the spinor exceptions, i.e., integers $u$ such that $Q$ does not represent the values of the quadratic sequence $ux^2$ for integers $x$; there are always only finitely many spinor exceptions (for an overview of their properties, see~\cite[pp.~351--352]{schulze-pillot_2000}).
	In particular, a spinor exception can exist only if the genus of $Q$ breaks into an even number of spinor genera. However, a necessary condition for this to happen is that the determinant of $Q$ is not squarefree~\cite[Ch.~11, Theorem~1.3]{cassels_1978}, whereas in our case, the determinant $abvp$ is squarefree.
	Alternatively, one can deduce the non-existence of spinor exceptions from the explicit results of Earnest, Hsia, and Hung~\cite{earnest_hsia_hung_1994}.
\end{proof}

We are now ready to prove Theorem~\ref{thm:main-almost}.

\begin{proof}[Proof of Theorem~\ref{thm:main-almost}] When $k=p=2$, then $\la 1,1,2\ra$ is $(2,1)$-universal.
	When $k=p$ is an odd prime, the theorem was proved  in Proposition \ref{prop:prime-almost-univ} for the form $\la 1,q,p\ra$ from Proposition \ref{prop:anisotropic}.
	
	If $k$ and $\ell$ are coprime, then there exists a prime $p$ such that $p\mid k$ and $p\nmid\ell$. Then $S_{k,\ell}\subset S_{p,\ell}$, and so every almost $(p,\ell)$-universal form is also almost $(k,\ell)$-universal.
	
	Finally, if $d=\gcd(k,\ell)$, then let $Q$ be an almost $(k/d,\ell/d)$-universal ternary form (which exists by the previous paragraph). Then $dQ$ is almost $(k,\ell)$-universal.
\end{proof}

\section{Non-existence of \texorpdfstring{$(p,\ell)$}{(p,l)}-universal forms}\label{sec:univ}

The reasoning behind Conjectures~\ref{conj:fin} and \ref{conj:2-101} is based on several observations from numerical experiments.
Prior to stating the observations, let us denote $X_{Q,p}$ the set of non-represented numbers (we call them \emph{gaps}) for a ternary form~$Q$ that is anisotropic precisely at $p$ (and $\infty$):
\[
	X_{Q,p} \eqdef \set{n\in \bN}{n\not\equiv 0\Pmod p,\,n\text{ not represented by }Q}
.\]
Note that when the determinant $d_Q$ is squarefree and $Q$ is anisotropic precisely at an odd prime $p$ and $\infty$, then $X_{Q,p}$ is finite by Proposition \ref{prop:prime-almost-univ}.

\bigskip

We carried out the following computations.

\subsection{Full search for \texorpdfstring{$(p,\ell)$}{(p,l)}-universal forms}\label{sec:no}

For a specific odd prime $p$ and $0<\ell<p$, there is an easy algorithm that searches for $(p,\ell)$-universal diagonal ternary forms.
Suppose that such a form exists.
Then there certainly exists a $(p,\ell)$-universal form $Q=\la a,b,c\ra$ such that:
\begin{itemize}
\item all three coefficients are squarefree (for instance, if $a=a'd^2$, then $\la a',b,c\ra$ is also $(p,\ell)$-universal);
\item $a\leq b\leq c$;
\item $a\leq \ell$ (one of the coefficients must be less than $\ell$ as the form represents $\ell$);
\item $\la a,b\ra$ represents $\ell$ (suppose it does not; then $c\leq\ell<p$ and this is a contradiction with $p\divides abc$);
\item if $\ell$ is not squarefree, then $a\leq\ell/2$ (either $\ell=ax^2$ and as $\ell$ is not squarefree and $a$ is, $a<\ell$, $x\geq2$, whence $a\leq\ell/4$; or $\ell=by^2$ and $a\leq b\leq\ell/4$; or $\ell=ax^2+by^2$, whence $2a\leq a+b\leq\ell$);
\item $p\divides b$ or $p\divides c$ (the prime divides the determinant $abc$ and as $a\leq\ell < p$, $p\ndivides a$).
\end{itemize}
Moreover, we know that no unary or binary $(p,\ell)$-universal forms exist.
So if we denote $e_{q}$ the smallest number $\equiv \ell\Pmod p$ not represented by a quadratic form $q$, we know that $b\leq e_{\la a\ra}$ and $c\leq e_{\la a,b\ra}$.

Based on these conditions, we searched through all the possible triples $(a,b,c)$ and obtained:

\begin{proposition}
For all primes $11\leq p\leq 1237$ and $1\leq \ell\leq p-1$ such that $(p,\ell)\neq(101,98)$, there are no diagonal ternary positive $(p,\ell)$-universal quadratic forms.
\end{proposition}

We carried out this computation in Python 2.7.12 on Intel Xeon machines with Ubuntu 16.04, kernel version 4.13.0.
This computation took 670 CPU days to complete.
We precomputed a list of primes and a list of squarefree numbers in SageMath~6.4~\cite{sagemath} as Python natively does not support these.

\subsection{\texorpdfstring{$(101,98)$}{(101,98)}-universality of \texorpdfstring{$\la1,2,101\ra$}{<1,2,101>}}

We did not manage to prove that the form $Q=\la1,2,101\ra$ represents all elements of $S_{101,98}$,
 as the form is not regular (the size of the genus is 9 as computed by Magma~\cite{magma}).
We assert the following:

\begin{proposition}
The form $\la1,2,101\ra$ represents $S_{101,98}\cap[0,10^{12}]$.
\end{proposition}

We carried out this computation in C\Rplus\Rplus\@ with GCC 4.8.3 on an Intel i5 PC.
To verify the result, we iterated through triples $x,y,z$ such that $x^2+2y^2\leq10^{12}$, $x^2+2y^2\equiv 98\Pmod{101}$
 and $x^2+2y^2+101z^2\leq10^{12}$.
However, there are simply too many such triples, so we first restricted to small $x$,
 and then interactively increased the $x$'s considered until we got representations of the whole $S_{101,98}\cap[0,10^{12}]$.

\subsection{Number of gaps}\label{sec:exp-ne}

When a form has a squarefree determinant and is anisotropic precisely at an odd prime $p$ and $\infty$, we know by Proposition \ref{prop:prime-almost-univ} that there are only finitely many numbers not represented by the form
 (as always excluding the zero class modulo $p$), i.e., the set of gaps $X_{Q,p}$ is finite.

We investigated the cardinality $\#X_{Q,p}$ for forms of small squarefree determinant $d_Q$ and we observe that it behaves very
 roughly as $p\log p$.
The ratio $\alpha:=\#X_{Q,p}/p\log p$ for forms with $p<300$ and $d_Q<30p$ is shown in Figure~\ref{fig:plogp}.

\begin{figure}
	\centering
	\includegraphics{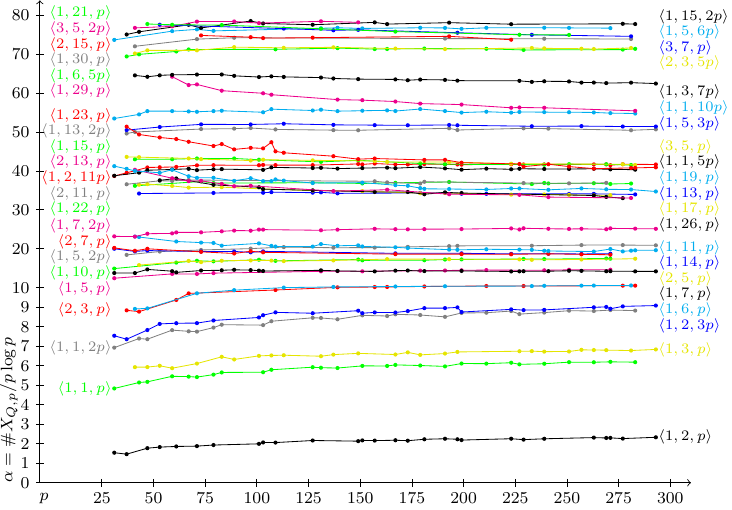}
	\caption{Comparison of a lower estimate of $\#X_{Q,p}$ to $p\log p$ for $30<p<300$ and $d_Q<30p$.
		Note that only forms with $\#X_{Q,p}<100p\log p$ are shown.}
	\label{fig:plogp}
\end{figure}

In the case of non-squarefree determinant, first of all note that without loss of generality
 (as in \S\,\ref{sec:no}) we can consider only forms with squarefree coefficients $a\leq b\leq c$.
Then there are two possibilities:
\begin{itemize}
\item First, if $p^2 \divides d_Q$, the form is $Q=\langle a, wp, vp\rangle$.
 Such a form can represent either quadratic residues or non-residues modulo $p$, and so $X_{Q,p}$ is always infinite.
 However, considering $X'_{Q,p} = X_{Q,p} \cap \{\text{represented classes}\allowbreak\ \text{$\operatorname{mod} p$}\}$,
 it turns out that $X' _{Q,p}$ behaves similarly to $X_{Q,p}$, in particular, the computations showed that $\#X'_{Q,p} > 100 p\log p$.
\item Second, if $p^2 \ndivides d_Q$, then $q^2 \divides d_Q$ for some prime $q\neq p$ and $X_{Q,p}$ may be finite (or infinite), but we got $\alpha>100$ always in these cases.
\end{itemize}
Hence from now on (except for the general Proposition \ref{prop:gaps}) we consider only squarefree determinants.

To estimate $\#X_{Q,p}$ (or $\#X'_{Q,p}$) we actually computed $X_{Q,p}\cap[0,120000p]$.
These computations were performed in C++ with GCC 5.4.0 on a cluster with Intel Xeon CPU cores @\,2.00--2.40\,GHz, and took less than 1~CPU day to complete. We precomputed the list of forms $Q$ anisotropic precisely at $30<p<300$ and with $d_Q<30p$ (or $d_Q<30p^2$) using SageMath~\cite{sagemath}.

In order to give these results a more theoretical basis, let us prove that $\#X_{Q,p}$ is at least of order $p\log p$.

\begin{proposition}\label{prop:gaps}
	Let $a,b,v$ be fixed positive integers such that $ab$ is squarefree. There is a constant $\alpha_0>0$ (depending only on $ab$) such that for all sufficiently large primes $p$, the ternary quadratic form $Q=\la a,b,vp\ra$ has at least $\alpha_0 vp\log p$ gaps, i.e., $\#X_{Q,p}>\alpha_0 vp\log p$.
\end{proposition}

Note that here we are not assuming that the determinant of $Q$ is squarefree, and so  $X_{Q,p}$ may be infinite.

\begin{proof}
	Let us consider the set of integers represented by the binary form $\la a,b\ra$.
	Bernays~\cite{bernays_1912} proved that there is a constant $\gamma_0$ (depending only on the product~$ab$) such that 
	\[
		\# \set{n\leq X}{n=ax^2+by^2} \sim \gamma_0\frac {X}{\sqrt{\log X}}
	\]
	(where the notation $f(X)\sim g(X)$ denotes the fact that $\lim_{X\rightarrow\infty}f(X)/g(X)=1$),
	cf.~\cite{blomer_granville_2006}.
	
	Choose $p$ large enough and let $P:=vp$ and $X:=\beta P\log P$ (where the constant $\beta$ will be specified later). Let us estimate from above the number of positive integers $n\leq X$ that are represented by the ternary form $Q=\la a,b,P\ra=ax^2+by^2+Pz^2$. For any such representation, we have $0\leq z\leq \sqrt{X/P}=\sqrt {\beta\log P}$ (without loss of generality, we assume here that $z\geq 0$) and $ax^2+by^2\leq X-Pz^2\leq X$.
	
	The number of possible values of $ax^2+by^2$ is $< \gamma X\log^{-1/2} X$ by Bernays's theorem (for any $\gamma>\gamma_0$ and sufficiently large $p$), and so the total number of 
	positive integers $n\leq X$ represented by $Q$ is 
	\[
		<\bigl(\sqrt {\beta\log P}+1\bigr) \gamma X\log^{-1/2} X<\sqrt\beta\gamma X+\gamma X\log^{-1/2} X
	.\]
	The second term is of smaller order than the first, and so if $\sqrt\beta\gamma<1$, then $Q$ does not represent a positive proportion of integers $\leq X$.
	
	More precisely, let us estimate the number of gaps $\#X_{Q,p}$ by counting only the gaps $\leq X$. As gaps are, by definition, not divisible by $p$, we just need to subtract the number of multiples of $p$ from the number of non-represented integers to obtain
	\begin{multline*}
		\#X_{Q,p}
		> X-\bigl(\sqrt\beta\gamma X+\gamma X\log^{-1/2} X\bigr) - (X/p+1)
	\\
		= \bigl(1-\sqrt\beta\gamma\bigr)\beta P\log P-(\gamma X\log^{-1/2} X+X/p+1)
	.\end{multline*}
	To maximize the first term, we can choose $\beta=\frac 4{9\gamma^2}$; then $(1-\sqrt\beta\gamma)\beta=\frac 4{27\gamma^2}$. The second term in the parenthesis is of smaller order of magnitude, which finishes the proof for any $\alpha_0<\frac 4{27\gamma_0^2}$.	
\end{proof}

In the argument below, we will need that $\alpha>7$ for almost all ternary forms. This is not possible to obtain by the preceding proof, as $\gamma_0$ is often too large. Nevertheless, the Proposition indicates that $\#X_{Q,p}$ is at least of the correct order, supplementing the numerical data from Figure~\ref{fig:plogp}. For illustration, we list estimates of $\gamma_0$ and $\alpha_0$ for small values of $ab$:
\[\text{\raisebox{\baselineskip}{
\begin{tabular}[t]{ccc}\toprule
{$ab$} & {$\gamma_0$} & {$\alpha_0$} \\\midrule
 \phantom01 &  0.79 &  0.24 \\
 \phantom02 &  0.90 &  0.18 \\
 \phantom03 &  0.66 &  0.34 \\
 \phantom05 &  0.56 &  0.48 \\
 \phantom06 &  0.58 &  0.44 \\
\bottomrule\end{tabular}
\ 
\begin{tabular}[t]{ccc}\toprule
{$ab$} & {$\gamma_0$} & {$\alpha_0$} \\\midrule
 \phantom07 &  0.56 &  0.47 \\
         10 &  0.49 &  0.61 \\
         11 &  0.54 &  0.52 \\
         13 &  0.44 &  0.78 \\
         14 &  0.49 &  0.62 \\
\bottomrule\end{tabular}
\ 
\begin{tabular}[t]{ccc}\toprule
{$ab$} & {$\gamma_0$} & {$\alpha_0$} \\\midrule
         15 &  0.39 &  0.98 \\
         17 &  0.45 &  0.74 \\
         19 &  0.46 &  0.72 \\
         21 &  0.33 &  1.41 \\
         22 &  0.39 &  0.98 \\
\bottomrule\end{tabular}
\ 
\begin{tabular}[t]{ccc}\toprule
{$ab$} & {$\gamma_0$} & {$\alpha_0$} \\\midrule
         23 &  0.45 &  0.73 \\
         26 &  0.43 &  0.83 \\
         29 &  0.40 &  0.93 \\
         30 &  0.31 &  1.55 \\
         31 &  0.41 &  0.88 \\
\bottomrule\end{tabular}
}}\]
(Note that the rate of convergence to Bernays' asymptotics is very slow \cite{blomer_granville_2006}, and so the numbers in the table should be only viewed as approximations for $\gamma_0$ and $\alpha_0$. For example, using the formula for the Landau-Ramanujan constant~\cite{finch_2003} we see that, for $ab=1$, we have $\gamma_0=0.76422\dots$ instead of $0.79$ shown in our table.)
These estimates were computed in Python 2.7.12 on a Intel Xeon CPU core @\,2.00--2.40\,GHz, and took 11~CPU days to complete. We estimate $\gamma$ by computing numbers represented by $\langle1,b\rangle$ up to $10^{11}$.

Further, note that for a form $\la a,b,vp\ra$ we have $\alpha>\alpha_0  v$ by Proposition \ref{prop:gaps}. Thus if $v$ is sufficiently large relative to $\alpha_0$ (which depends only on $ab$), then we indeed have $\alpha>7$.

\subsection{Heuristic argument}\label{sec:heur}

If we now assume that the elements of $X_{Q,p}$ are equidistributed modulo $p$,
 we can use the comparison of $\#X_{Q,p}$ with $p\log p$ in the following heuristic argument:
For a specific form $Q$, denote $\alpha>0$ such constant that $\#X_{Q,p}=\alpha p\log p$ (see Figure~\ref{fig:plogp}).
For $Q$ to be $(p,\ell)$-universal for specific $0<\ell<p$, we need that \emph{none} of the gaps lies in the set $S_{p,\ell}$. 
The probability of this is 
\[
	\Bigl(1-\frac{1}{p-1}\Bigr)^{\#X_{Q,p}}
	= \Bigl(1-\frac{1}{p-1}\Bigr)^{\alpha p\log p}
	\approx e^{-\alpha\frac{p}{p-1}\log p}
	\approx p^{-\alpha}
.\]
Then the expected number of $\ell$'s such that $Q$ is $(p,\ell)$-universal is $(p-1)p^{-\alpha} \approx p^{1-\alpha}$.
This shows that the larger the value of $\alpha$, the smaller the chance that a form is $(p,\ell)$-universal for some $\ell$.

We can even use this to estimate the total expected number of $(p,\ell)$-universal forms:
Oh \cite[Theorem 2.3]{oh_2011_ijnt} proved an upper bound for the determinant of a $(p,\ell)$-universal ternary, which implies $d_Q< Cp^6$ for a constant $C$.
As $p\mid d_Q$, there are at most $Cp^5$ possible determinants.
Thus asymptotically, for each $p$ there are
at most $p^{5+\varepsilon}$ candidates for $(p,\ell)$-universal forms, as the number of ways of factoring $d_Q=abc$ is below $d_Q^\varepsilon$ for any $\varepsilon$.
Therefore, for a fixed $p$ the expected number of $\ell$'s and $Q$'s such that $Q$ is $(p,\ell)$-universal is
asymptotically smaller than $p^{6+\varepsilon-\min\alpha}$.

Let us exclude the forms $\la1,1,p\ra$, $\la1,2,p\ra$, and $\la1,3,p\ra$ for now. The data behind Figure~\ref{fig:plogp} suggest that we eventually have $\alpha>7$ for each of the remaining forms. Thus for given sufficiently large $p$, the expected number of (non-excluded) $(p,\ell)$-universal forms is less than $C'p^{-1-\varepsilon'}$ for some $\varepsilon'>0$ and a constant $C'$. The total expected number over all large primes $p$ is then $C'\sum_p p^{-1-\varepsilon'}$, which \emph{converges}! Thus besides from the excluded  forms $\la1,1,p\ra$, $\la1,2,p\ra$, and $\la1,3,p\ra$, we expect to have only finitely many $(p,\ell)$-universal ones.

Unfortunately, it turns out that $X_{Q,p}$ is not equidistributed modulo $p$. The distribution appears to be not too far from normal (and in fact, seems to be skewed in favor of even fewer universal forms), and so the preceding heuristic computation still provides non-trivial information, especially since most of the values $\alpha$ are much larger than $\alpha>7$ that we needed.

Further, the above consideration shows an important aspect:
 the form $\la 1,2,p\ra$ is by orders more likely to be $(p,\ell)$-universal than any other form.
Thus we performed yet another experiment.

\subsection{\texorpdfstring{$(p,\ell)$}{(p,l)}-universality
 of \texorpdfstring{$\la 1,2,p\ra$}{<1,2,p>}}\label{sec:12p}

In order for $\la 1,2,p\ra$ to be $(p,\ell)$-universal, it has to be anisotropic precisely at $p$ (and $\infty$),
 which happens if and only if $p\equiv 5,7\Pmod8$.
For these primes, we calculated $X_{Q,p}$ (or rather a subset of it, namely $X_{Q,p}\cap[0,120000p]$)
 using the same software as in \S\,\ref{sec:exp-ne}
 and checked whether it contains elements from all classes $\not\equiv0\Pmod p$.
The computation took 290 CPU days to complete.
It turns out that $\la 1,2,p\ra$ is not $(p,\ell)$-universal for any $103\leq p<30000$.
And not only that; we even observe that classes containing small number of elements are extremely rare,
 as can be seen in Figure~\ref{fig:small}.

In Figure~\ref{fig:small}, primes $300<p<30000$ with $p\equiv5,7\Pmod8$ are shown, with each bar corresponding to a group of 100 primes.
For each group and $m<10$, we show (in shades of blue) the number of $(p,\ell)$
 such that our (lower) estimate of $X_{Q,p}\cap S_{p,\ell}$ equals $m$.
For comparison, we show (in gray) the total number of $\ell$'s.
Because there are very few $(p,\ell)$ for which $m\leq 5$, we highlight these (in red and green) in the bottom chart. 
Note that $m\leq1$ never appears for $p\geq103$ and that $m=2$ appears only for $p\leq1181$ and $p=6607$.

\begin{figure}\label{fig 2}
\centering
\includegraphics{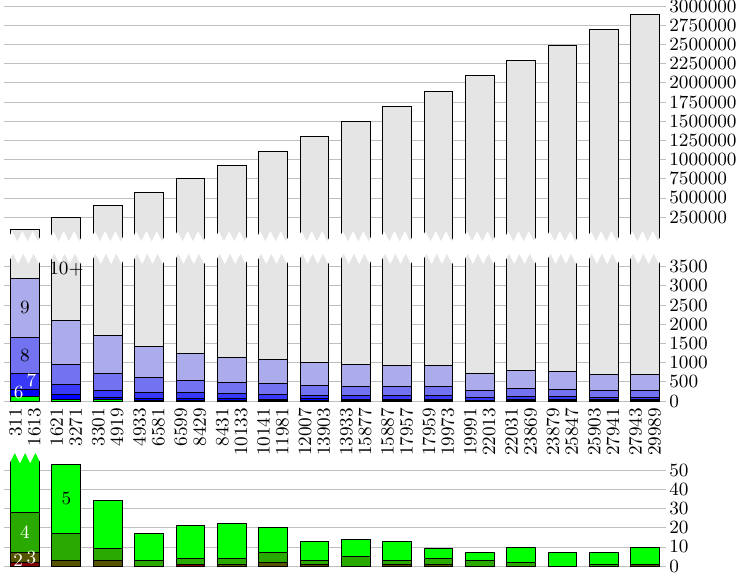}
\caption{Number of $(p,\ell)$'s such that $X_{Q,p}\cap [0,120000p]\cap S_{p,\ell} = m$ for $m=0,1,\dots,9$.
 For details, see \S\,\ref{sec:12p}.}
\label{fig:small}
\end{figure}

Of course, similar computations can be done for other forms with small $\alpha$, such as $\la 1,1,p\ra$ or $\la 1,3,p\ra$.
We tested $\la 1,1,p\ra$ and $\la 1,3,p\ra$ for primes $p<15000$ getting no candidates for $(p,\ell)$-universal forms (other than $p=2,3$). This computation took 210~CPU days to complete.

\subsection{Conclusion of the experiments}

We can summarize our observations from the computations as follows:
\begin{enumerate}
\item there is strong evidence that only finitely many diagonal ternary forms $Q\neq \la 1,2,p \ra, \la 1,1,p \ra$, $\la 1,3,p \ra$
 are $(p,\ell)$-universal (see \S\,\ref{sec:heur});
\item $Q=\la 1,2,p\ra$ is not $(p,\ell)$-universal for any $103<p<30000$, and also
the cases such that the set of gaps $X_{Q,p}\cap S_{p,\ell}$ is small rapidly cease to exist as $p$ increases
 (see \S\,\ref{sec:12p}). Likewise for $\la 1,1,p \ra, \la 1,3,p \ra$ (that are not $(p,\ell)$-universal for $3<p<15000$).
\end{enumerate}

This convinces us that only finitely many $(p,\ell)$-universal diagonal ternary forms exist,
which is the claim of Conjecture~\ref{conj:fin}.
Conjecture~\ref{conj:2-101} is then motivated by the absence of $(p,\ell)$-universal forms for $103\leq p<1257$ (see \S\,\ref{sec:no}) and the non-universality of $\la 1,2,p\ra$ up to $30000$.

\section*{Acknowledgments}

We thank Valentin Blomer and Pavlo Yatsyna for useful discussions that improved the article, and the anonymous referees for several helpful comments.

\bibliographystyle{amsalpha}
\bibliography{biblio}

\providecommand{\bysame}{\leavevmode\hbox to3em{\hrulefill}\thinspace}
\providecommand{\MR}{\relax\ifhmode\unskip\space\fi MR }
\providecommand{\MRhref}[2]{%
  \href{http://www.ams.org/mathscinet-getitem?mr=#1}{#2}
}
\providecommand{\href}[2]{#2}
\begin{thebibliography}{EHH94}

\bibitem[Ber12]{bernays_1912}
Paul Bernays, \emph{\"{U}ber die {D}arstellung von positiven, ganzen {Z}ahlen
  durch die primitiven, bin\"{a}ren quadratishen {F}ormen einer
  nicht-quadratishen {D}iskriminante [{O}n presentation of positive whole
  numbers via primitive binary quadratic forms with non-quadratic
  discriminant]}, Dissertation, G\"{o}ttingen, 1912.

\bibitem[BG06]{blomer_granville_2006}
Valentin Blomer and Andrew Granville, \emph{Estimates for representation
  numbers of quadratic forms}, Duke Math. J. \textbf{135} (2006), no.~2,
  261--302.

\bibitem[Cas78]{cassels_1978}
J.~W.~S. Cassels, \emph{Rational quadratic forms}, London Mathematical Society
  Monographs, vol.~13, Academic Press, Inc. [Harcourt Brace Jovanovich,
  Publishers], London-New York, 1978.

\bibitem[DSP90]{duke_schulze-pillot_1990}
William Duke and Rainer Schulze-Pillot, \emph{Representation of integers by
  positive ternary quadratic forms and equidistribution of lattice points on
  ellipsoids}, Invent. Math. \textbf{99} (1990), no.~1, 49--57.

\bibitem[EHH94]{earnest_hsia_hung_1994}
Andrew~G. Earnest, John~S. Hsia, and David~C. Hung, \emph{Primitive
  representations by spinor genera of ternary quadratic forms}, J. London Math.
  Soc. (2) \textbf{50} (1994), no.~2, 222--230.

\bibitem[Fin03]{finch_2003}
Steven~R. Finch, \emph{Mathematical constants}, Encyclopedia of Mathematics and
  its Applications, vol.~94, Cambridge University Press, Cambridge, 2003.

\bibitem[Han04]{hanke_2004}
Jonathan Hanke, \emph{Some recent results about (ternary) quadratic forms},
  Number theory, CRM Proc. Lecture Notes, vol.~36, Amer. Math. Soc.,
  Providence, RI, 2004, pp.~147--164.

\bibitem[Jag96]{jagy_1996}
William~C. Jagy, \emph{Five regular or nearly-regular ternary quadratic forms},
  Acta Arith. \textbf{77} (1996), no.~4, 361--367.

\bibitem[JKS97]{jagy_kaplansky_shiemann_1997}
William~C. Jagy, Irving Kaplansky, and Alexander Schiemann, \emph{There are 913
  regular ternary forms}, Mathematika \textbf{44} (1997), no.~2, 332--341.

\bibitem[Kap95]{kaplansky_1995}
Irving Kaplansky, \emph{Ternary positive quadratic forms that represent all odd
  positive integers}, Acta Arith. \textbf{70} (1995), no.~3, 209--214.

\bibitem[Kne61]{kneser_1961}
Martin Kneser, \emph{Darstellungsmasse indefiniter quadratischer {F}ormen},
  Math. Z. \textbf{77} (1961), 188--194.

\bibitem[LO14]{lemke_2014}
Robert~J. Lemke~Oliver, \emph{Representation by ternary quadratic forms}, Bull.
  Lond. Math. Soc. \textbf{46} (2014), no.~6, 1237--1247.

\bibitem[\MA97]{magma}
\MAGMAX{Wieb Bosma, John Cannon, and Catherine Playoust}, \emph{The {M}agma
  algebra system. {I}. {T}he user language}, J. Symbolic Comput. \textbf{24}
  (1997), no.~3-4, 235--265, Magma V2.24-5 Calculator at
  \url{http://magma.maths.usyd.edu.au/calc/} [2019-03-12].

\bibitem[Oh11a]{oh_2011_aa}
Byeong-Kweon Oh, \emph{Regular positive ternary quadratic forms}, Acta Arith.
  \textbf{147} (2011), no.~3, 233--243.

\bibitem[Oh11b]{oh_2011_ijnt}
\bysame, \emph{Representations of arithmetic progressions by positive definite
  quadratic forms}, Int. J. Number Theory \textbf{7} (2011), no.~6, 1603--1614.

\bibitem[OS97]{ono_soundararajan_1997}
Ken Ono and K.~Soundararajan, \emph{Ramanujan's ternary quadratic form},
  Invent. Math. \textbf{130} (1997), no.~3, 415--454.

\bibitem[PW18]{pehlivan_williams_2018}
Lerna Pehlivan and Kenneth~S. Williams, \emph{{$(k,l)$}-universality of ternary
  quadratic forms {$ax^2+by^2+cz^2$}}, Integers \textbf{18} (2018), Paper No.
  A20, 44.

\bibitem[Rou14]{rouse_2014}
Jeremy Rouse, \emph{Quadratic forms representing all odd positive integers},
  Amer. J. Math. \textbf{136} (2014), no.~6, 1693--1745.

\bibitem[\SA14]{sagemath}
\SAGEX{The Sage Developers}, \emph{{S}agemath, the {S}age {M}athematics
  {S}oftware {S}ystem ({V}ersion 6.4)}, 2014, \url{https://www.sagemath.org}.

\bibitem[SP00]{schulze-pillot_2000}
Rainer Schulze-Pillot, \emph{Exceptional integers for genera of integral
  ternary positive definite quadratic forms}, Duke Math. J. \textbf{102}
  (2000), no.~2, 351--357.

\bibitem[Sun17]{sun_2017}
Zhi-Wei Sun, \emph{Tuples $(m,r,a,b,c)$ with $30 \geq m > \max\{2,r\} \geq 0$
  and $100 \geq a \geq b \geq c > 0$, for which all the numbers $mn+r$ ($n =
  0,1,2,\dotsc$) should be representable by $ax^2+by^2+cz^2$ with $x$, $y$, $z$
  integers}, 2017, in OEIS: The On-Line Encyclopedia of Integer Sequences.
  \url{https://oeis.org/A286885/a286885_1.txt} [2021-05-12].

\bibitem[WS18]{wu_sun_2018}
Hai-Liang Wu and Zhi-Wei Sun, \emph{Arithmetic progressions represented by
  diagonal ternary quadratic forms}, 2018, preprint, 16~pp.,
  \href{http://arxiv.org/abs/1811.05855}{\ttfamily arXiv:1811.05855}.

\end{thebibliography}

\end{document}